\documentclass[11pt, reqno]{amsart}
\usepackage{amsmath,amssymb,amsthm}
\usepackage{txfonts,mathrsfs}
\usepackage[utf8]{inputenc}
\usepackage[english]{babel}
\usepackage{enumerate}
\usepackage{graphicx}
\usepackage{color}
\usepackage{tikz-cd}
\usepackage{bbm}
\usepackage{csquotes}

\usepackage[backend=bibtex,giveninits=true,maxbibnames=99, doi=false,isbn=false,url=false,eprint=true]{biblatex}

\bibliography{refs}

\usepackage[hidelinks,
  bookmarks=true,
  bookmarksopen=true,
  bookmarksnumbered=true,
  pdfauthor={Giuliano Basso},
  pdftitle={Absolute Lipschitz extendability and linear projection constants},
  pdfsubject={Paper},
  pdfstartview=FitH,
  colorlinks=false,
  pdfencoding=unicode
]{hyperref}

\numberwithin{equation}{section}

\newtheorem{theorem}{Theorem}[section]
\newtheorem{proposition}[theorem]{Proposition}
\newtheorem{lemma}[theorem]{Lemma}

\theoremstyle{definition}
\newtheorem{remark}[theorem]{Remark}
\newtheorem{definition}[theorem]{Definition}

\newtheoremstyle{customNumber}
     {}          
     {}          
     {\itshape}  
     {}          
     {\bfseries} 
     {.}         
     { }         
     {\thmname{#1}\thmnumber{ #2}\thmnote{ #3}}
\theoremstyle{customNumber}

\renewcommand{\phi}{\varphi}

\DeclareMathOperator{\R}{\mathbb{R}}
\DeclareMathOperator{\N}{\mathbb{N}}
\DeclareMathOperator{\Z}{\mathbb{Z}}
\newcommand*\norm[1]{\lVert#1\rVert}

\newcommand\abs[1]{\lvert#1\rvert}

\DeclareMathOperator{\Lip}{Lip}

\DeclareMathOperator{\Tr}{tr}

\DeclareMathOperator{\fin}{fin}
\DeclareMathOperator{\absolutE}{\text{\normalfont \ae}}
\DeclareMathOperator{\diam}{diam}
\DeclareMathOperator{\sep}{sep}
\DeclareMathOperator{\spann}{span}

\DeclareMathOperator{\ext}{ext}
\DeclareMathOperator{\dual}{dual}
\DeclareMathOperator{\id}{id}
\DeclareMathOperator{\sgn}{sgn}

\DeclareMathAlphabet{\mathcal}{OMS}{cmsy}{m}{n}

\begin{document}

\title{Absolute Lipschitz extendability and linear projection constants}

\author{Giuliano Basso}
\address{Department of Mathematics\\ University of Fribourg\\ Chemin du Mus\'ee 23\\ CH-1700 Fribourg \\ Switzerland}
\email{giuliano.basso@unifr.ch}

\keywords{Lipschitz extension, projection constant, linear programming}

\subjclass[2020]{Primary 54C20; Secondary 46B20 and 46E99}

\thanks{Research supported by Swiss National Science Foundation Grant no. 182423}

\begin{abstract}
We prove that the absolute extendability constant of a finite metric space may be determined by computing relative projection constants of certain Lipschitz-free spaces. As an application, we show that \(\mbox{ae}(3)=4/3\) and \(\absolutE(4)\geq (5+4\sqrt{2})/7\). Moreover, we discuss how to compute relative projection constants by solving linear programming problems.
\end{abstract}

\maketitle

\section{Introduction}

\subsection{Background}

In the setting of non-linear geometry of Banach spaces one considers Banach spaces as metric spaces and studies non-linear analogues of linear notions. Sometimes it turns out that such a non-linear analogue is completely recoverable from its linear counterpart.  For example, a fundamental result due to Lindenstrauss (see \cite[Theorem 5]{lindenstrauss1964}) states that
a dual Banach space is a \(C\)-absolute Lipschitz retract if and only if it is a \(C\)-absolute \textit{linear} retract. Another result in this direction concerning simultaneous Lipschitz extensions has been obtained by Brudnyi and Brudnyi in \cite{MR2340707}.
Their result is restated as Theorem~\ref{thm:BB} below. In this paper, we obtain a slight generalization of Brudnyi and Brudnyi's result and use it to relate the absolute extendability constant \(\absolutE(X)\) of a finite metric space \(X\) to relative projection constants \(\lambda(\mathcal{F}(X), \mathcal{F}(Y))\). Here, \(Y\) is a finite metric space containing \(X\) and \(\mathcal{F}(Y)\) denotes the Lipschitz-free space over \(Y\).  As an application, we show that \(\absolutE(3)=4/3\) and the lower bound
\(\absolutE(4)\geq (5+4\sqrt{2})/7\),
which we conjecture to be sharp. 

Before presenting our results in more detail, we first introduce some standard terminology used in the context of quantitative Lipschitz extension problems.
We consider the following diagram:

\begin{center}
\begin{tikzcd}
Y  \arrow[dashrightarrow]{dr} \\
X \arrow[u, hook] \arrow{r}{f} & E.
\end{tikzcd}
\end{center}

Here, \(Y\) is a metric space, \(X\subset Y\) a subset endowed with the induced metric, \(E\) a Banach space over \(\R\) and \(f\colon X\to E\) a Lipschitz map. The infimum of those \(K\geq 0\) for which there exists a \(K\Lip(f)\)-Lipschitz map \(\bar{f}\colon Y\to E\) making the diagram above commutative is denoted by \(e(X,Y,E,f)\). Here, we use the notation \(\Lip(f)\coloneqq \inf\bigl\{ L\geq 0 : \text{ \(f\) is \(L\)-Lipschitz}\bigr\}\). 

There is an abundance of literature discussing `Lipschitz extension problems'. The reader may refer to the monographs \cite{MR2882877, MR2868143, MR3931701} for a recent account of the theory. 
In this paper, we are interested in `Lipschitz extension problems' of the following form:

\begin{itemize}
\item \textit{Trace problems}:
fix \(X\) and \(Y\), and vary everything else: 
\[\nu(X,Y)\coloneqq\sup\big\{e(X,Y,E,f) : E, f \textrm{ arbitrary}\big\}.\]
\item \textit{Absolute Lipschitz extendability}: fix \(X\) and vary everything else:
\[\absolutE(X)\coloneqq\sup\big\{ \nu(X,Y) : Y \textrm{ arbitrary} \big\}.
\]
\end{itemize} 
The absolute extendability constant \(\absolutE(X)\) is finite for a wide variety of metric spaces (see \cite[Corollary 5.2]{MR2834732} due to Naor and Silberman). For example, Lee and Naor (see \cite[Theorem 1.6]{lee2005extending}) proved that there is a universal constant \(C>0\) such that if \(X\) is doubling with doubling constant \(N_X\), then \(\absolutE(X) \leq C \log(N_X)\).
In particular, by setting 
\[
\absolutE(n)\coloneqq\sup\big\{\absolutE(X) : \abs{X}=n \big\},
\]
one has \(\absolutE(n) \leq C \log(n)\). Naor and Rabani \cite[Theorem 1]{naor2017lipschitz} (lower bound), and Lee and Naor \cite[Theorem 1.10]{lee2005extending} (upper bound), improved this estimate by showing that there are constants \(c\), \(C>0\) such that for every \(n\geq 3\),
\begin{equation}\label{eq:NaorLeee}
c \sqrt{\log(n)} \leq \absolutE(n) \leq C \frac{\log(n)}{\log(\log(n))}. 
\end{equation}
These are the best known bounds of \(\absolutE(\cdot)\). In contrast to these strong asymptotic estimates,
up to the author's knowledge, the only known exact
values of \(\absolutE(\cdot)\) are 
\(\absolutE(1) = 0 \text{ and } \absolutE(2) = 1\). Theorem \ref{thm:main2} below yields a formula of \(\textrm{\ae}(n)\) involving only \textit{linear} Lipschitz extension moduli. Using that the exact values of some of these linear Lipschitz extension moduli are known, we can add \(\textrm{\ae}(3)=\tfrac{4}{3}\) to the sequence above; see Proposition~\ref{prop:ae3} below.
 
For the remainder of this subsection, let us briefly discuss the quantity \(\nu(X,Y)\) and its relation to simultaneous Lipschitz extension. Let \((X, p)\) be a pointed metric space and denote by \(\Lip_0(X)\) the real Banach space of all real-valued Lipschitz functions on \(X\) with \(f(p)=0\) equipped with the norm
\[
\abs{f}_{\Lip}\coloneqq \sup_{x\neq x'} \frac{d(f(x), f(x'))}{d(x,x')}.
\]
We refer to Weaver's book  \cite{doi:10.1142/9911} for a survey on  \(\Lip_0(X)\). A \textit{linear extension operator} for \(X\subset Y\) is a bounded linear map \(T\colon \Lip_0(X)\to \Lip_0(Y)\) such that \(T f|_X=f\) for all \(f\in \Lip_0(X)\). We use the notation:
\[
\lambda_{\Lip}(X,Y)\coloneqq \inf\bigl\{\norm{T} \,:\, T \text{ is a linear extension operator for \(X\subset Y\)}\bigr\}.
\]
and
\[
\lambda_{\Lip}(Y) \coloneqq \sup\big\{ \lambda_{\Lip}(X,Y) \,:\, X\subset Y\big\}.
\]
Surprisingly, due to a result of Brudnyi and Brudnyi (see \cite[Theorem 1.2]{MR2340707}), there is a formula for \(\lambda_{\Lip}(Y)\) using only non-linear Lipschitz extension constants \(\nu_{\fin}(X,Y)\). 
By putting
\[
\nu_{\mathcal{B}}(X,Y)\coloneqq \sup\big\{ e(X,Y,E,f)\, : E\in \mathcal{B}, f\colon X\to E \text{ Lipschitz map} \big\},
\]
where \(\mathcal{B}\) is a class of Banach spaces, and setting 
\[
\nu_{\fin}(Y)\coloneqq\sup_{X\subset Y} \nu_{\mathcal{B}_{\fin}}(X,Y),
\]
where \(\mathcal{B}_{\fin}\) denotes the class of all finite-dimensional Banach spaces,
their result can be stated as follows:
\begin{theorem}[Brudnyi and Brudnyi \cite{MR2340707}]\label{thm:BB}
For every metric space \(Y\) the following identity is true:
\[
\lambda_{\Lip}(Y)=\nu_{\fin}(Y).
\]
\end{theorem} 
 Lower and upper bounds of \(\lambda_{\Lip}(Y)\) for many interesting classes of metric spaces, such as  Gromov-hyperbolic groups, \(\R\)-tress, certain Riemannian manifolds, and classical Banach spaces have been obtained by Brudnyi and Brudnyi \cite{MR2288741} and Naor \cite{MR3627773}. Hence, by Theorem~\ref{thm:BB}, the quantity \(\nu_{\fin}(Y)\) can be estimated
for members of any such family of metric spaces.

\subsection{Main results}
Our first result is the following variant of Theorem~\ref{thm:BB}:

\begin{theorem}\label{thm:main1}
Let \(X\subset Y\) denote metric spaces and \(\mathcal{B}\) a class of Banach spaces. We define \(K_{\mathcal{B}}\coloneqq\sup\{ \nu(E, E^{\ast\ast}) : E\in \mathcal{B}\}\). 
Then 
\begin{equation}\label{eq:upper1}
 \nu_{\mathcal{B}}(X,Y) \leq K_{\mathcal{B}} \cdot \sup_{F} \lambda_{\Lip}(F, Y),
\end{equation}
where the supremum is taken over all finite subsets \(F\subset X\). Moreover, if \(\Lip_0(X)^\ast\) is contained in \(\mathcal{B}\), then 
\begin{equation}\label{eq:lower-111}
\lambda_{\Lip}(X,Y) \leq \nu_{\mathcal{B}}(X,Y).
\end{equation}
As a consequence of \eqref{eq:upper1} and \eqref{eq:lower-111}, 
\begin{equation}\label{eq:fin-dual}
\lambda_{\Lip}(Y)=\nu_{\dual}(Y)=\nu_{\fin}(Y).
\end{equation}
\end{theorem}
Here, we use the notation \(\nu_{\dual}(Y)\) to denote \(\nu_{\mathcal{B}_{\dual}}(Y)\), where \(\mathcal{B}_{\dual}\) is the class consisting of all dual Banach spaces. Our proof of Theorem~\ref{thm:main1} is a streamlined version of Brudnyi and Brudnyi's proof of Theorem~\ref{thm:BB}. The introduction of \(\nu_{\dual}(Y)\) makes it possible to obtain the identity \(\lambda_{\Lip}(Y)=\nu_{\fin}(Y)\) as a direct consequence of \(\lambda_{\Lip}(Y)=\nu_{\dual}(Y)\); see also Remark~\ref{rem:bb}. Using Theorem~\ref{thm:main1}, we obtain the following estimate:
\begin{equation}\label{eq:ae1}
\absolutE(X)\geq \sup_{Y\supset X} \lambda_{\Lip}(X, Y),
\end{equation}
where the supremum is taken over all metric spaces \(Y\) containing \(X\). For finite metric spaces \(X\), \eqref{eq:ae1} is an equality; see Theorem~\ref{thm:main2}.

The quantity \(\lambda_{\Lip}(X,Y)\) is closely related to the relative projection constant of the Lipschitz-free spaces of \(X\) and \(Y\). To state this relationship precisely we need to recall some concepts from Banach space theory. 
Every point \(x\in X\) induces a linear functional \(\delta(x)\colon \Lip_0(X)\to \R\) via \(\delta(x)(f)\coloneqq f(x)\). One can define the Lipschitz-free space of \(X\) as follows:

\begin{definition}\label{def:lipschitz-free}
Let \((X,p)\) be a pointed metric space. The Lipschitz-free space of \(X\), denoted by \(\mathcal{F}(X)\), is the closure of \(\spann\{\delta(x) : x\in X\}\) in \(\Lip_0(X)^\ast\).
\end{definition}

By construction, \(\mathcal{F}(X)\) is a Banach space over \(\R\). Lipschitz-free spaces (also called Arens-Eells spaces or transportation cost spaces) have been introduced by Arens and Eells in the 1950s (see \cite{MR81458}). The term `Lipschitz-free space' has been coined by Godefroy and Kalton in \cite{MR2030906}. It follows directly from the definition of \(\mathcal{F}(X)\) that the map \(\delta_X\colon X\to \mathcal{F}(X)\) defined by \(x\mapsto \delta(x)\) is an isometric embedding. 
We will often need the following universal property of Lipschitz-free spaces. Whenever \((X, p)\) is a pointed metric space and \(f\colon X\to E\) is a Lipschitz map into a Banach space satisfying \(f(p)=0\), there exists a unique linear map \(\beta_f\colon \mathcal{F}(X)\to E\), such that \(\beta_f\circ \delta_X=f\). Moreover, one has \(\norm{\beta_f}=\Lip(f)\).

 Using this universal property, one can show that whenever \(\iota \colon X \to Y\)  is a base-point preserving isometric embedding, then there exists a unique linear isometric embedding \(\hat{\iota}\colon \mathcal{F}(X) \to \mathcal{F}(Y)\) satisfying \(\hat{\iota} \circ \delta_X=\delta_Y \circ \iota\). In fact, one necessarily has \(\hat{\iota}=\beta_{\delta_Y \circ \iota}\) and by using McShane's extension theorem (see, for example, \cite[Theorem 1.33]{doi:10.1142/9911}) it is easy to see that \(\hat{\iota}\) is distance preserving. Hence, if \(X\subset Y\), there is a canonical way to consider \(\mathcal{F}(X)\) as a subspace of \(\mathcal{F}(Y)\). 

Given two Banach spaces \(E\subset F\), the \textit{linear projection constant \(\lambda(E,F)\) of \(E\) relative to \(F\)} is by definition the infimum of the norms of all linear projections from \(F\) onto \(E\).  Projection constants have a rich history in Banach space theory. We refer to the books \cite{MR2300779, MR993774, MR902804} and the references therein for some classical results on projection constants. By the above, there is a canonical way to consider \(\mathcal{F}(X)\) as a subspace of \(\mathcal{F}(Y)\) whenever \(X\subset Y\), and therefore the projection constant \(\lambda(\mathcal{F}(X), \mathcal{F}(Y))\) is well-defined. If \(X\) is finite and \(Y\) is any metric space, then 
\begin{equation}\label{eq:dictionary}
\lambda_{\Lip}(X,Y)=\lambda\bigl(\mathcal{F}(X), \mathcal{F}(Y)\bigr).
\end{equation}
This equality is proven in Lemma~\ref{lem:translate-lip-to-proj}; see also \cite[Lemma 3.2]{MR2340707}. A standard argument shows that whenever \(X\) is a finite metric space then \(\nu(X,Y)\) is less than or equal to the supremum of \(\nu(X, F)\) taken over all finite subsets \(F\subset Y\) which contain \(X\); see Lemma~\ref{lem:reduction-to-finite-sets}.
Hence, by combining this with \eqref{eq:ae1}, \eqref{eq:dictionary}
and invoking Lemma~\ref{lem:finishLine}, we obtain the following formula for \(\absolutE(X)\) when \(X\) is a finite metric space.

\begin{theorem}\label{thm:main2}
For every finite metric space \(X\), 
\begin{equation}\label{eq:main-theorem14}
\absolutE(X)=\sup_{Y} \lambda\bigl(\mathcal{F}(X), \mathcal{F}(Y)\bigr),
\end{equation}
where the supremum is taken over all finite metric spaces \(Y\) containing \(X\).
Moreover,  \begin{equation}\label{eq:simple-consequence}
\absolutE(X)=\lambda\bigl(\mathcal{F}(\iota(X)), \mathcal{F}(c_0)\bigr)
\end{equation}
for every isometric embedding \(\iota\colon X\to c_0\). 
\end{theorem}

Recall that \(\mathcal{F}(\iota(X))\) is naturally identified with \(\hat{\iota}(\mathcal{F}(X))\subset \mathcal{F}(c_0)\), and so \eqref{eq:simple-consequence} reads as follows: \(\absolutE(X)=\lambda\bigl(\hat{\iota}(\mathcal{F}(X)), \mathcal{F}(c_0)\bigr)\). 
There is a subtlety involving this identity.
In fact, if \(L\colon \mathcal{F}(X)\to \mathcal{F}(c_0)\) is any linear isometric embedding, then \(\absolutE(X)=\lambda\bigl(\hat{\iota}(\mathcal{F}(X)), \mathcal{F}(c_0)\bigr)\) is not necessarily true when \(\hat{\iota}(\mathcal{F}(X))\subset \mathcal{F}(c_0)\) is replaced with its isometric copy \(L(\mathcal{F}(X))\subset \mathcal{F}(c_0)\).

This can be seen as follows.
The unit ball of \(\mathcal{F}(X)\) is equal to the closed convex hull of the linear functionals \(m_{xy}=d(x,y)^{-1}\big(\delta(x)-\delta(y)\big),\, x\neq y\). Thus, if \(X\) is finite, then the unit ball of \(\mathcal{F}(X)\) is a convex polytope and  there exists a linear isometric embedding \(R\colon \mathcal{F}(X)\to c_0\). It is well-known (see, for example, \cite[Theorem III.B.5]{MR1144277}) that for any finite-dimensional subspace \(E\subset c_0\), one has \(\lambda(E, c_0)=\lambda(E)\), where \(\lambda(E)\) is the \textit{absolute projection constant} of \(E\), that is, 
\[
\lambda(E)=\sup\{ \lambda(E, F) : E\subset F\}.
\]
Hence, \(\lambda\big(R(\mathcal{F}(X)), c_0\big)=\lambda\big(\mathcal{F}(X)\big)\). Now, since \(c_0\) is separable, a deep result of Godefroy and Kalton (see \cite[Corollary 3.3]{MR2030906}) tells us that there is a linear isometric embedding \(S\colon c_0\to \mathcal{F}(c_0)\).

By setting \(L\coloneqq S\circ R\) we find, by the above, that \(\lambda(L(\mathcal{F}(X)), \mathcal{F}(c_0))=\lambda(\mathcal{F}(X))\). 
As we will see below, there are finite metric spaces \(X\) for which \(\absolutE(X)<\lambda(\mathcal{F}(X))\). Hence, for any such space the conclusion of Theorem~\ref{thm:main2} is not true if \(\hat{\iota}(\mathcal{F}(X))\subset \mathcal{F}(c_0)\) is replaced with the isometric copy \(L(\mathcal{F}(X))\subset \mathcal{F}(c_0)\) with \(L\) constructed as above.

\subsection{Applications}\label{sec:applications} 
As a direct consequence of Theorem \ref{thm:main2},  \begin{equation}\label{eq:trivial-estimate}
\absolutE(X) \leq \lambda(\mathcal{F}(X))
\end{equation}
for any finite metric space \(X\).
For polyhedral finite-dimensional Banach spaces \(E\) the exact value of \(\lambda(E)\) can be computed by solving a linear programming problem (see Lemma~\ref{lem:lin-prog} and the remark thereafter). Hence, as \(\mathcal{F}(X)\) polyhedral, \eqref{eq:trivial-estimate} gives a numeric upper bound of \(\absolutE(X)\) whenever the distance matrix of \(X\) is given.

In general the upper bound \eqref{eq:trivial-estimate} is far from being sharp. Indeed, for a finite weighted tree \(T\), Godard (see \cite[Corollary 3.6]{MR2680057}) proved that \(\mathcal{F}(T)\) is linearly isometric to \(\ell_1^n\), with \(n\coloneqq\abs{T}-1\); thus a result of Gr\"unbaum (see \cite[Theorem 3]{grunbaum1960projection}) tells us that for such a weighted tree \(T\) with \((n+1)\in 2\Z\) vertices, the right hand side of \eqref{eq:trivial-estimate} equals
\[
\frac{n \,\Gamma\bigl(n/2\bigr)}{\sqrt{\pi}\,\Gamma\big( (n+1)/2\big)} \text{\large \(\sim\)} \sqrt{\frac{2n}{\pi}}.
\]
Therefore, \eqref{eq:trivial-estimate} tells us that \(\absolutE(T)\leq C\sqrt{n}\). 
But, for any finite weighted tree \(T\), one has \(\absolutE(T)=1\), see Remark~\ref{rem:weighted-tree}, and we find that \eqref{eq:trivial-estimate} is clearly not sharp.

In what follows, we determine the exact value of \(\absolutE(3)\). From \eqref{eq:trivial-estimate}, we obtain
\begin{equation}\label{eq:crude}
\absolutE(n) \leq \lambda_{n-1},
\end{equation}
where \(\lambda_{n}\) denotes \textit{the maximal projection constant of order \(n\)}, that is,
\begin{equation}\label{eq:maximal-proj}
\lambda_n\coloneqq\sup\{ \lambda(E) : \dim(E)=n\}.
\end{equation}
Due to an important result of Kadets and Snobar (see \cite{MR291770}), one has \(\lambda_n\leq \sqrt{n}\) for all \(n\geq 1\). 
The maximal projection constant \(\lambda_n\) is difficult to compute, the only known values are \(\lambda_1=1\) and \(\lambda_2=\tfrac{4}{3}\), the former due to the Hahn-Banach theorem and the latter due to Chalmers and Lewicki \cite{BruceL2010}. 
In \cite{Konig1985}, K\"onig proved that \(\lambda_n \geq \sqrt{n}-1\) for a subsequence of integers \(n\). Hence, by taking into account Lee and Naor's upper bound of \(\absolutE(n)\), we find that \eqref{eq:crude} is not sharp for \(n\geq 1\) large enough. 
But for \(n=3\) the inequality \eqref{eq:crude} is in fact an equality:
\begin{proposition}\label{prop:ae3}
\[\absolutE(3)=\frac{4}{3}.\]
\end{proposition}

We suspect that inequality \eqref{eq:crude} is strict already for \(n=4\).
 Our next result bounds the quantity \(\absolutE(4)\). 
\begin{proposition}\label{prop:ae4}
\[\frac{5+4\sqrt{2}}{7} \leq \absolutE(4) \leq\frac{3+6 \sqrt{2}}{7}.\]
\end{proposition}
The upper bound is obtained in two steps. First, we show that if \(X\) consists of four points, then there exists an eight-point metric space \(Y\), such that \(X\subset Y\) and \(\absolutE(X)=\lambda\bigl(\mathcal{F}(X), \mathcal{F}(Y)\bigr)\). This is done by considering the injective hull of \(X\).  As a result, \(\absolutE(4) \leq \bar{\lambda}(3,7)\), where
\begin{equation}\label{eq:def-of-bar-lambda}
\bar{\lambda}(n,d)\coloneqq\max\bigl\{ \lambda(E, F) : E\subset F, \, \dim(E)=n \text{ and } \dim(F)=d \bigr\}.    
\end{equation}
Now, the upper bound of \(\absolutE(4)\) follows from the estimate
\begin{equation}\label{eq:koenig}
\bar{\lambda}(n,d)\leq \frac{n}{d}+\sqrt{\bigl(d-1\bigr)\frac{n}{d}\bigl(1-\frac{n}{d}\bigr) },
\end{equation}
which is due to K\"onig, Lewis, and Lin (see \cite{konig1983finite}). 

Thanks to Theorem~\ref{thm:main2}, the lower bound in Proposition \ref{prop:ae4} follows from a  computation showing that \(\lambda(\mathcal{F}(X_0), \mathcal{F}(Y_0))\geq (5+4\sqrt{2})/7\), where 
\[
X_0\coloneqq\Bigl\{ \bigl(0,0\bigr), \bigl(1+\sqrt{2},0\bigr), \bigl(0,2\bigr), \bigl(1+\sqrt{2},2\bigr) \Bigr\}\subset \ell_\infty^2,
\] 
the vertices of a rectangle with side lengths \(2\) and \(\sqrt{2}+1\) considered as a subset of \(\ell_\infty^2=(\R^2, \norm{\cdot}_{\infty})\). In \cite{CHALMERS2009553}, Chalmers and Lewicki  showed that \(\lambda(3,5)=(5+4\sqrt{2})/7\), where
\begin{equation}\label{eq:maximal-relative-proj}
\lambda(n,d)\coloneqq\max\bigl\{ \lambda(E, \ell_\infty^d) : \text{\(E\) is an \(n\)-dimensional subspace of \(\ell_\infty^d\)} \bigr\}
\end{equation}
for all \(n,d\geq 1\) with \(n\leq d\). Numerical simulations strongly suggest that the lower bound in Proposition \ref{prop:ae4} is sharp and thus \(\absolutE(4)=\lambda(3,5)\). This is particularly intriguing as \(\lambda_2=\lambda(2,3)\) and thus \(\absolutE(3)=\lambda(2,3)\) due to Proposition \ref{prop:ae3}. We do not know if this is part of a general pattern. 

Up to \(d=6\), the exact values of \(\lambda(n,d)\) are known (see \cite[Section 1.4]{MR4001080}) and numerical lower bounds have been calculated in \cite[p.326]{MR3566474} up to \(d=10\). Moreover,
in \cite[Lemma 2.6]{CHALMERS2009553}, it is shown that \(\lambda(d-1,d)=2-2/d\)  whenever \(d\geq 2\). Thus, on account of \eqref{eq:koenig}, it follows that \(\bar{\lambda}(d-1, d)\) is equal to \(\lambda(d-1,d)\). But, in general, it seems to be an open question whether \(\bar{\lambda}(n,d)=\lambda(n,d)\) for all \(n,d\geq 1\) with \(n\leq d\).

As our last application of Theorem \ref{thm:main2} we establish an upper bound of \(\absolutE(X)\) for arbitrary metric spaces. We put \(\sep X\coloneqq \inf\{ d(x,x') : x, x'\in X, x\neq x'\}\). The upper bound 
\[
\absolutE(X)\leq 2\frac{\diam X}{\sep X}
\]
has been obtained by Johnson, Lindenstrauss and Schechtman in \cite[p. 138]{MR852474}. We can strengthen their estimate as follows:

\begin{proposition}\label{prop:bounds-ae-sep}
For every metric space \(X\), 
\[\absolutE(X)\leq  2\Big(1-\frac{1}{\abs{X}}\Big)\frac{\diam X}{\sep X}. \]
\end{proposition}
We use the convention that \(\tfrac{a}{0}=\infty\), \(\tfrac{a}{\infty}=0\) and \(a\cdot \infty=\infty\) for all \(a>0\). 

\subsection{Acknowledgements}
Parts of this work are contained in the author’s PhD thesis \cite{20.500.11850/398970}. I am thankful to Urs Lang for helpful discussions. Moreover, I am indebted to the anonymous reviewer for a very helpful report from which the article benefited greatly.

\section{Preliminaries}

\subsection{Injective metric spaces}

A metric space \(X\) is called \textit{injective} if whenever \(A\subset B\) are metric spaces and \(f\colon A\to X\) is a \(1\)-Lipschitz map, then there exists a \(1\)-Lipschitz map \(\bar{f}\colon B\to X\) extending \(f\). It is well-known that a metric space is injective if and only if it is an absolute \(1\)-Lipschitz retract, or if and only if it is hyperconvex (see, for example, \cite[Propositions 2.2 and 2.3]{lang2013}). Examples of injective metric spaces include \(\ell_\infty(I)\), for any index set \(I\), and complete \(\R\)-trees (see \cite[Proposition 2.1]{lang2013}). It is easy to check that if \(X\) is injective, then \(\absolutE(X)=1\). 

The following result goes back to Isbell \cite{isbell}. For every metric space \(X\) there exists an injective metric space \(E(X)\) such that \(X\subset E(X)\) and if \(f\colon E(X)\to Y\)  is a \(1\)-Lipschitz map for which \(f|_X\) is an isometric embedding, then the map \(f\) is an isometric embedding. Thus, if \(X\subset Y\) and \(Y\) is injective, then \(X\subset E(X) \subset Y\) and \(E(X)\) may be interpreted as the `smallest' injective metric space containing \(X\). The space \(E(X)\) is called \textit{injective hull} of \(X\). Equivalent characterizations of the injective hull can be found in \cite[Proposition 3.4]{lang2013}.

For finite metric spaces \(X\) the injective hull \(E(X)\) is a finite-dimensional polyhedral complex having only finitely many isometry types of cells. The cells are subsets of \(\ell_\infty^{n}\), where \(n\) is the greatest integer such that \(2n\leq \abs{X} \). For a recent survey of injective hulls with applications to geometric group theory we refer to Lang's article \cite{lang2013}. 
Our interest in injective metric spaces stems from the following simple lemma.

\begin{lemma}\label{lem:reduction-to-inj-hull}
If \(X\) is a metric space, then  \(\absolutE(X)=\nu(X,Y)\) for every injective metric space \(Y\) containing \(X\). 
In particular, \(\absolutE(X)=\nu(X, E(X))\). 
\end{lemma}

\begin{proof}
Let \(Y\supset X\) be an injective metric space and \(Y'\supset X\) any metric space. As \(Y\) is injective, the identity map \(i\colon X\to X\) admits a \(1\)-Lipschitz extension \(\bar{i}\colon Y'\to Y\).
Fix \(\varepsilon >0\) and set \(K_\varepsilon\coloneqq \nu(Y, X)+\varepsilon\).
Suppose now that \(f\colon X\to E\) is an \(L\)-Lipschitz map to a Banach space \(E\).
By definition of \(\nu(X, Y)\), there is \(K_\varepsilon L\)-Lipschitz map \(\bar{f}\colon Y\to E\) extending \(f\).
The composition \(\bar{f}\circ\bar{i}\) is an \(K_\varepsilon L\)-Lipschitz extension of \(f\) to \(Y'\). As \(\varepsilon>0\) war arbitrary, we conclude that \(\nu(X, Y')\leq \nu(X,Y)\). This completes the proof. 
\end{proof}

\begin{remark}\label{rem:weighted-tree}
In what follows, we show that \(\absolutE(T)=1\) whenever \(T\) is a finite weighted tree equipped with the shortest-path metric \(d_T\), which is defined in \eqref{eq:shortest-path}.    
Suppose that \(T=(X, E, \omega)\) is a finite weighted tree with positive weights. Then it is not hard to see that  \(X_T\coloneqq (X, d_T)\) is a \(0\)-hyperbolic metric space (such spaces are also called tree-like metric spaces). Hence, by a result of Dress (see \cite[Theorem 8]{MR753872}), it follows that \(Y\coloneqq E(X_T)\) is a complete \(\R\)-tree. Given \(y\), \(y'\in Y\) we denote by \([y,y']_Y\subset Y\) the unique geodesic connecting \(y\) and \(y'\) and we abbreviate \((y,y')_Y\coloneqq [y, y']_Y\setminus \{y,y'\}\). By applying \cite[Lemma 2.3]{basso2021approximating}, we obtain
\begin{equation}\label{eq:darstellung-Y}
Y=\bigcup_{x, x'\in X} [x,x']_Y.
\end{equation}
Since \(Y\) is uniquely geodesic and \(X_T\subset Y\), for every shortest path \((x_0, \dots, x_k)\) in \(T\) one has \([x_0, x_k]_Y=[x_0, x_1]_Y\cup \dotsm \cup [x_{k-1}, x_k]_Y\) and \((x_i, x_{i+1})_Y\cap (x_{j}, x_{j+1})_Y=\varnothing\)  for all distinct \(i, j=0,\dots, k-1\). Hence, as all edges \(\{x_1, x_1'\}\), \(\{x_2, x_2'\}\in E\) are contained in some shortest path in \(T\), we obtain that \((x_1,x_1')_Y\cap (x_2, x_2')_Y=\varnothing\) whenever \(\{x_1, x_1'\}\neq \{x_2, x_2'\}\). Consequently, by the use of \eqref{eq:darstellung-Y}, 
\begin{equation}\label{eq:disjoint-union-1}
Y=X_T\cup \dot{\bigcup_{\{x, x'\}\in E}} (x,x')_Y.
\end{equation}
Now, suppose that \(f\colon X_T\to E\) is an \(L\)-Lipschitz map. We define \(\bar{f}\colon Y\to E\) as follows. If \(y\in [x,x']_Y\) with \(\{x,x'\}\in E\), then we put \(\bar{f}(y)=(1-t)f(x)+t f(x')\), where \(t\coloneqq d(y, x)/d(x,x')\). Because of \eqref{eq:disjoint-union-1}, \(\bar{f}\) is well-defined. By construction, \(\bar{f}|_X=f\) and a short computation reveals that \(\bar{f}\) is \(L\)-Lipschitz. This proves that \(\nu(X_T, Y)=1\), and so by virtue of Lemma~\ref{lem:reduction-to-inj-hull}, it follows that \(\absolutE(X_T)=1\), as desired.  
\end{remark}

\subsection{Lipschitz extension}
The following lemma is standard. It can be obtained by a straightforward application of a variant of McShane's extension theorem. 
\begin{lemma}\label{lem:co-is-finitely-injective}
Every \(1\)-Lipschitz map \(f\colon F\to c_0\) from a finite subset \(F\) of a metric space \(X\) admits a \(1\)-Lipschitz extension \(\bar{f}\colon X\to c_0\). 
\end{lemma}
\begin{proof}
 For each \(i\in \N\) let \(\pi_i\colon c_0\to \R\) denote the \(i\)th coordinate projection.
Using McShane's extension theorem (see, for example, \cite[Theorem 1.33]{doi:10.1142/9911}), for every \(i\in \N\) we find a \(1\)-Lipschitz extension \(\bar{f}_i\colon X\to \R\) of the map \(f_i\coloneqq \pi_i \circ f\) such that \(\norm{\bar{f}_i}_{\infty}=\norm{f_i}_\infty\). Since \(F\) is finite, it follows that \(\norm{f_i}_\infty \to 0\) as \(i\to \infty\). Hence, the map \(\bar{f}\colon X\to c_0\) given by \(x\mapsto (\bar{f}_i(x))\) is well-defined and a \(1\)-Lipschitz extension of \(f\), as desired. 
\end{proof}
 A map \(\varrho\colon X\to Y\) is called \textit{\(C\)-bilipschitz}, \(C\geq 1\), if there is a real number \(s>0\) such that
\[
s \,d(x,x') \leq d(\varrho(x), \varrho(x')) \leq C s \,d(x,x')
\]
for all points \(x\), \(x'\in X\). Two metric spaces \(X_1\) and \(X_2\) are called \textit{\(C\)-bilipschitz equivalent} if there exists a \(C\)-bilipschitz bijection \(\varrho\colon X_1\to X_2\). 
The following lemma is employed in the proof Proposition~\ref{prop:bounds-ae-sep}. Its proof boils down to a simple argument involving injective hulls.
\begin{lemma}\label{lem:bilip-esti}
One has
\[
\absolutE(X_1) \leq C \absolutE(X_2)
\]
whenever the metric spaces \(X_1\) and \(X_2\) are \(C\)-bilipschitz equivalent.
\end{lemma}

\begin{proof}
Suppose that \(X_1\) and \(X_2\) are \(C\)-bilipschitz equivalent via the bijection \(\varrho\colon X_1\to X_2\).   Since \(E(X_2)\) is injective, the map \(\varrho\colon X_1\to X_2\) admits a Lipschitz extension \(\bar{\varrho}\colon E(X_1)\to E(X_2)\) satisfying \(\Lip(\bar{\varrho})=\Lip(\varrho)\). Fix \(\varepsilon>0\) and set \(K_\varepsilon\coloneqq \nu(X_2, E(X_2))+\varepsilon\). Let \(f\colon X_1\to E\) be an \(L\)-Lipschitz map to a Banach space \(E\).  By definition of \(\nu(X_2, E(X_2))\), there exists a \(K_\varepsilon \Lip(g)\)-Lipschitz extension \(\bar{g}\colon E(X_2)\to E\) of the map \(g\coloneqq f \circ \varrho^{-1}\). We set \(\bar{f}\coloneqq \bar{g}\circ \bar{\varrho}\). The map \(\bar{f}\) extends \(f\) and we have
\[
\Lip(\bar{f})\leq \Lip(\bar{g})\Lip(\bar{\varrho}) \leq K_\varepsilon \Lip(f) \Lip(\varrho)\Lip(\varrho^{-1}).
\]
Notice that \(\Lip(\varrho)\Lip(\varrho^{-1})\leq C\). Hence, \(\Lip(\bar{f})\leq C K_\varepsilon \Lip(f)\). Consequently, we obtain 
\[
\nu(X_1, E(X_1))\leq C\big(\nu(X_2, E(X_2))+\varepsilon\big).
\]
By Lemma~\ref{lem:reduction-to-inj-hull}, \(\absolutE(X_1)=\nu(X_1, E(X_1))\) and \(\absolutE(X_2)=\nu(X_2, E(X_2))\). As \(\varepsilon>0\) was arbitrary, we infer \(\absolutE(X_1)\leq C \absolutE(X_2)\), as desired. 
\end{proof}

\subsection{Lipschitz-free spaces}
In what follows, we collect some elementary facts about Lipschitz-free spaces. We refer to the books \cite{doi:10.1142/9911} and \cite{MR3114782} for additional information on Lipschitz-free spaces. 
We will often use the following well-known universal property of Lipschitz-free spaces (see, for example, \cite[Theorem 3.6]{doi:10.1142/9911}).  

\begin{lemma}\label{lem:universal-property}
Let \((X, p)\) be a pointed metric space. If \(f\colon X\to E\) is a Lipschitz map into a Banach space satisfying \(f(p)=0\), then there exists a unique linear map \(\beta_f\colon \mathcal{F}(X)\to E\) such that \(\beta_f\circ \delta_X=f\). Moreover, one has \(\norm{\beta_f}=\Lip(f)\).
\end{lemma}
It is a simple consequence of the above that for any two points \(p\), \(p'\in X\) the Lipschitz-free spaces over \((X, p)\) and \((X, p')\) are linearly isometric.
Moreover, by considering \(E=\R\), it also follows directly from Lemma~\ref{lem:universal-property} that \(\Lip_0(X)\to \mathcal{F}(X)^\ast\) defined by \(f\mapsto \beta_f\) is a linear isometry. If \((X,p)\) is a pointed metric space consisting of \(n\geq 2\) points, then \(\Lip_0(X)\) is isomorphic \(\R^{n-1}\) in the sense of \(\R\)-vector spaces. Thus, as \(\Lip_0(X)\) is linearly isometric to \(\mathcal{F}(X)^\ast\), it follows that \(\mathcal{F}(X)\) is \((n-1)\)-dimensional and thus \(\{ \delta(x) : x\in X \text{ and } x\neq p\}\) is a basis of \(\mathcal{F}(X)\). Clearly, \(\delta(p)=0\). 

We close this section by stating two important results involving \(\mathcal{F}(X)\) when \(X\) is a finite metric space or a finite weighted tree, respectively. In Section~\ref{sec:three}, we need the following key fact.
\begin{lemma}\label{lem:translate-lip-to-proj}
One has
\[
\lambda_{\Lip}(X, Y)=\lambda(\mathcal{F}(X), \mathcal{F}(Y))
\]
whenever \(X\subset Y\) is a finite subset of a metric space \(Y\). 
\end{lemma}
\begin{proof}
Let \(T\colon \Lip_0(X)\to \Lip_0(Y)\) be a linear extension operator for \(X\subset Y\).
The adjoint \(T^\ast\colon \Lip_0(Y)^\ast \to \Lip_0(X)^\ast\) of \(T\) satisfies \(\norm{T^\ast}=\norm{T}\). Since \(X\) is finite-dimensional, we have \(\Lip_0(X)^\ast=\mathcal{F}(X)\), and so the restriction \(P\coloneqq T^\ast|_{\mathcal{F}(Y)}\) is a linear projection from \(\mathcal{F}(Y)\) onto \(\mathcal{F}(X)\). By construction, \(\norm{P}\leq \norm{T}\) and therefore \(\lambda(\mathcal{F}(X), \mathcal{F}(Y))\leq\lambda_{\Lip}(X, Y)\). To see the other inequality, notice that whenever \(P\colon \mathcal{F}(Y)\to \mathcal{F}(X)\) is a linear projection, then \(T\coloneqq P^\ast\) is a linear extension operator for \(X\subset Y\). Consequently, \(\lambda_{\Lip}(X,Y)\leq \lambda(\mathcal{F}(X), \mathcal{F}(Y))\). 
\end{proof}

Let \(T=(X, E, \omega)\) be a finite weighted tree with positive weights. We denote by \(d_T\colon X\times X \to \R\) the shortest-path metric on \(T\) induced by \(\omega\), that is, \(d_T(x,x)=0\) for all \(x\in X\) and for all distinct \(x\), \(x'\in X\),
\begin{equation}\label{eq:shortest-path}
d_T(x,x')=\sum_{i=0}^{k-1} \omega(\{x_i, x_{i+1}\})
\end{equation}
where \((x_0, \dots, x_k)\) is the shortest-path in \(T\) from \(x\) to \(x'\). Fix a basepoint \(p\in X\).
By abuse of notation, we write \(\mathcal{F}(T)\) to denote the Lipschitz-free space of the pointed metric space \((X, d_T, p)\). In \cite[Corollary 3.6]{MR2680057}, Godard proved that \(\mathcal{F}(T)\) is isometric to \(\ell_1^{\abs{T}-1}\). 
In Section~\ref{sec:four}, we need the following explicit construction of such an isometry.

\begin{lemma}\label{lem:explicit-isometry}
Let \(T=(X, E, \omega)\) be a finite weighted tree with positive weights. Fix a basepoint \(p\in X\), an enumeration \(\{f_1, \dots, f_N\}\) of \(E\), and \(\varepsilon_i\in \{-1,1\}\) for all \(i=1, \dots, N\). Let \(f\colon X\to \ell_1^N\) be defined by \(f(x)=(\xi_1, \dots, \xi_N)\) with
\[
\xi_i=
\begin{cases}
\varepsilon_i\,\omega_i & \text{ if \(f_i\) is an edge of the shortest path from \(p\) to \(x\)} \\
0 &\text{otherwise}.
\end{cases} 
\]
Then \(\beta_f\colon \mathcal{F}(T)\to \ell_1^N\) is a linear isometry if \(T\) is equipped with the shortest-path metric \(d_T\).
\end{lemma}
\begin{proof}
Notice that \(\norm{f(x)}_1=d_T(x,p)\) for all \(x\in X\). Thus, as \(T\) is a tree, it is readily verified that \(f\) is an isometric embedding if \(X\) is equipped with \(d_T\). Hence, \(\norm{\beta_f}=1\). Notice that if \(f_i=\{x,x'\}\) then \(e_i\in \ell_1^N\) is contained in the linear span of \(\beta_f(\delta(x)-\delta(x'))\). This implies that \(\beta_f\) is bijective. Letting \(\alpha_f\coloneqq\beta_f^{-1}\), it remains to show that \(\norm{\alpha_f}=1\). 
As \(f\) is an isometric embedding it is easy to check that \(\norm{\alpha_f(e_i)}=1\) for all \(i=1, \dots, N\). Consequently, for all \(v=(v_1, \dots, v_N)\in \ell_1^N\), one has \(\norm{\alpha_f(v)}=\norm{v_1\alpha_f(e_1)+\dotsm+v_N\alpha_f(e_N)}\leq \abs{v_1}+\dotsm +\abs{v_N}=\norm{v}_1\), and so \(\norm{\alpha_f}=1\), as desired.
\end{proof}

\section{Linear and non-linear Lipschitz extension moduli}\label{sec:three}

In this section, we prove Theorems~\ref{thm:main1} and \ref{thm:main2} from the introduction. 
The following lemma relates linear and non-linear Lipschitz extension moduli when the source space \(X\) is finite. Lemma~\ref{lem:finishLine} is the main tool in our proof of Theorem~\ref{thm:main1}. 

\begin{lemma}\label{lem:finishLine}
Let \(Y\) be a metric space and \(X\subset Y\) a finite subset. Fix a basepoint \(p\in X\).
Then we have 
\[e(X,Y, \mathcal{F}(X), \delta_X)=\nu(X, Y)=\lambda_{\Lip}(X,Y)=\mathrm{\lambda}(\mathcal{F}(X), \mathcal{F}(Y)).\]
\end{lemma}
\begin{proof}
Recall that \(\delta\coloneqq \delta_X\colon X\to \mathcal{F}(X)\) is an isometric embedding. We set \(K\coloneqq e(X,Y,\mathcal{F}(X), \delta)\), that is, \(K\) denotes the infimum of those \(K^\prime \geq 1\) for which there exists a Lipschitz map \(\bar{\delta}\colon Y\to \mathcal{F}(X)\) extending \(\delta\) and satisfying \(\Lip(\bar{\delta}) \leq K^\prime\). By definition, \(K\leq\nu(X, Y)\). 
Due to Lemma~\ref{lem:universal-property}, for every Lipschitz map \(f\colon X\to E\) with \(f(p)=0\), the map \(\beta_f\colon \mathcal{F}(X)\to E\) satisfies \(f=\beta_f\circ \delta\) and \(\norm{\beta_f}=\Lip(f)\), and so \(e(X,Y,E)\leq e(X, Y, \mathcal{F}(X), \delta)=K\). Therefore, we infer \(\nu(X, Y)=K\). Next, we show that 
\begin{equation}\label{eq:chain-1}
\lambda_{\Lip}(X, Y)\leq K \leq \lambda(\mathcal{F}(X), \mathcal{F}(Y)).
\end{equation}
Since \(Y\subset \mathcal{F}(Y)\), we see that \(K\leq \lambda(\mathcal{F}(X), \mathcal{F}(Y))\). Now, let \(\bar{\delta}\colon Y\to \mathcal{F}(X)\) be a Lipschitz extension of \(\delta\). 
Let \(T\colon \Lip_0(X)\to \Lip_0(Y)\) denote the adjoint of \(\beta_{\bar{\delta}}\). By construction, \(T\) is a linear extension operator for \(X\subset Y\) and \(\norm{T}=\norm{\beta_{\bar{\delta}}}=\Lip(\bar{\delta})\). Hence, as \(\bar{\delta}\) was arbitrary, we obtain \(\lambda_{\Lip}(X,Y)\leq K\) and thereby \eqref{eq:chain-1} follows. By Lemma~\ref{lem:translate-lip-to-proj}, \(\lambda_{\Lip}(X,Y)=\lambda(\mathcal{F}(X), \mathcal{F}(Y))\). This completes the proof.
\end{proof}
We put
\[
e(X,Y,E)\coloneqq\sup\{ e(X,Y,E,f) : \text{\(f\) arbitrary}\}.
\]
The following lemma is well established. Variants of it appear at various places
in the mathematical literature (see, for example, \cite[Theorem 5]{lindenstrauss1964},  \cite[Lemma 1.1]{ball1992markov} or \cite[p. 168]{MR3108872}). 

\begin{lemma}\label{lem:reduction-to-finite-sets}
Let \(X\subset Y\) denote metric spaces and \(E\) a Banach space.
Then 
\begin{equation}\label{eq:very-helpful}
e(X, Y, E)\leq \nu(E, E^{\ast  \ast}) \, \sup_{F\subset Y} e(F\cap X, F, E^{\ast\ast}),
\end{equation}
where the supremum is taken over all finite subsets \(F\subset Y\).
Moreover, if \(X\) is finite, then
\begin{equation}\label{eq:very-helpful-2}
\nu(X, Y)\leq \sup_{X\subset F\subset Y} \nu(X, F),
\end{equation}
where the supremum is taken over all finite subsets \(F\subset Y\) with \(X\subset F\). 
\end{lemma}

\begin{proof}
We follow closely the proof given in \cite[Lemma 1.1]{ball1992markov}, which is due to Ball.
Another approach is sketched in \cite[p. 168]{MR3108872}. We abbreviate 
\[
K\coloneqq\sup \big\{e(F\cap X, F, E^{\ast\ast}) : F\subset Y \text{ finite subset}\big\}.
\]
Fix a point \(p\in X\) and let \(f\colon X\to E\) be a Lipschitz map. Without loss of generality, we may suppose that \(f\) is \(1\)-Lipschitz and \(f(p)=0\). For each point \(x\in X\) we define the topological space
\[B_x\coloneqq\big\{ v\in E^{\ast\ast} : \norm{v} \leq K d(x,p)\big\}\]
endowed with the weak-\(^\ast\) topology,  
and we set
\[B\coloneqq\prod_{x\in X} B_x.\]
Let \(J\colon E\to E^{\ast\ast}\) denote the canonical embedding of \(E\) into \(E^{\ast\ast}\).
For each finite subset \(F \subset Y\) with 
\(p\in F\), there exists an extension \(\bar{f}_{F} \colon F \to E^{\ast\ast}\) of the map \((J \circ f)|_{F\cap X}\) such that \(\Lip(\bar{f}_{F}) \leq K \).  We define the the point \(z_{F}\in B\) via
\[(z_{F})_x=\begin{cases}
\bar{f}_{F}(x) &  \textrm{ if } x\in F, \\
0 &\textrm{ otherwise.}
\end{cases}\]
Since each \(B_x\) is compact, Tychonoff's theorem tells us that the net \((z_{F})\), where \(F \subset Y\) is a finite subset that contains \(p\), has a subnet converging to some \(z\in B\). It is not hard to check that \(\bar{f}\colon Y \to E^{\ast\ast}\) defined by \(x\mapsto z_x\) is a \(K\)-Lipschitz extension of \(J \circ f\). Fix \(\varepsilon>0\). By definition of \(\nu(E, E^{\ast\ast})\), there is a projection \(q\colon E^{\ast\ast}\to J(E)\) with \(\Lip(q)\leq \nu(E, E^{\ast\ast})+\varepsilon\).
Consequently, the map \(q\circ \bar{f}\) is a \(K(\nu(E, E^{\ast\ast})+\varepsilon) \)-Lipschitz extension of \(f\). As \(\varepsilon>0\) was arbitrary, this gives \eqref{eq:very-helpful}. Suppose now that \(X\) is finite. We set
\[
K'\coloneqq\sup \big\{e(X, F, \mathcal{F}(X)^{\ast \ast}) : F\subset Y \text{ finite subset with \(X\subset F\)}\big\}.
\]
By exactly the same reasoning as above, for any \(1\)-Lipschitz map \(f\colon X\to \mathcal{F}(X)\) with \(f(p)=0\) the map \(J\circ f\) admits an \(K'\)-Lipschitz extension \(\bar{f}\colon Y\to \mathcal{F}(X)^{\ast\ast}\). Since \(X\) is finite, it follows that the canonical embedding \(J\colon \mathcal{F}(X)\to \mathcal{F}(X)^{\ast \ast}\) is an isometric isomorphism, and so \(e(X, Y, \mathcal{F}(X))\leq K'\). Now,  Lemma~\ref{lem:finishLine} tells us that \(e(X,Y, \mathcal{F}(X))=\nu(X,Y)\) and \(e(X, F, \mathcal{F}(X)^{\ast \ast})=\nu(X,F)\). Hence, by the above, \eqref{eq:very-helpful-2} follows. This completes the proof. 
\end{proof}

Now we are in position to prove Theorem~\ref{thm:main1}.

\begin{proof}[Proof of Theorem~\ref{thm:main1}]
As
\[
e(F\cap X, F, E^{\ast\ast})\leq \nu(F\cap X,F)\leq \nu(F\cap X, Y),
\]
the upper estimate \eqref{eq:upper1} is a direct consequence of Lemmas \ref{lem:finishLine} and \ref{lem:reduction-to-finite-sets}. Fix a basepoint \(p\in X\) and suppose now that \(\Lip_0(X)^\ast\) is contained in \(\mathcal{B}\). 
Let \(\bar{\delta}\colon Y\to \Lip_0(X)^\ast\) be a Lipschitz extension of the evaluation map \(\delta\colon X\to \Lip_0(X)^\ast\) and 
\(T\colon \Lip_0(X)^{\ast\ast}\to \Lip_0(Y)\) the adjoint of \(\beta_{\bar{\delta}}\). By construction, \(T|_{\Lip_0(X)}\) is a linear extension operator for \(X\subset Y\) and \(\norm{T|_{\Lip_0(X)}} \leq \norm{T}=\norm{\beta_{\bar{\delta}}}=\Lip(\bar{\delta})\). Consequently, 
\[
\lambda_{\Lip}(X, Y) \leq e(X,Y,\Lip_0(X)^\ast, \delta) \leq \nu_{\mathcal{B}}(X,Y),
\]
as desired. Thus we are left to establish \eqref{eq:fin-dual}. Notice that \(\nu(E, E^{\ast\ast})=1\) for every dual space \(E\). Indeed, if \(E=F^\ast\), then the adjoint \(J^\ast\colon F^{\ast\ast\ast} \to F^\ast\) of the canonical linear embedding \(J\colon F\to F^{\ast\ast}\) is a norm-one projection of \(E^{\ast\ast}\) onto \(E\). Now, on account of \eqref{eq:upper1}, we infer
\begin{equation}\label{eq:one}
\nu_{\dual}(Y)\leq \sup_{\substack{F\subset Y \\}} \lambda_{\Lip}(F, Y) \leq \lambda_{\Lip}(Y),
\end{equation}
where the supremum is taken over all finite subsets \(F\subset Y\).
Since \(\Lip_0(X)^\ast\) is contained in \(\mathcal{B}_{\dual}\), we have \(\nu_{\dual}(X, Y) \geq \lambda_{\Lip}(X, Y)
\) for every subset \(X\subset Y\). Hence, 
\begin{equation}\label{eq:two}
\nu_{\dual}(Y)\geq \lambda_{\Lip}(Y).
\end{equation}
By combining \eqref{eq:one} with \eqref{eq:two}, we conclude 
\begin{equation}\label{eq:three}
\lambda_{\Lip}(Y)=\sup_{\substack{F\subset Y \\ }}\lambda_{\Lip}(F, Y)=\nu_{\dual}(Y).
\end{equation}
Since \(\Lip_0(F)^\ast\in \mathcal{B}_{\fin}\) for every finite subset \(F\subset Y\), by \eqref{eq:lower-111}, we have \(\nu_{\fin}(F, Y)\geq \lambda_{\Lip}(F,Y)\), and we find
\[
\nu_{\dual}(Y)\geq \nu_{\fin}(Y)\geq \sup_{F\subset Y} \nu_{\fin}(F, Y) \geq \sup_{F\subset Y} \lambda_{\Lip}(F,Y).
\]
Because of \eqref{eq:three}, this implies that \(\nu_{\dual}(Y)=\nu_{\fin}(Y)\), as was to be shown. 
\end{proof}

\begin{remark}\label{rem:bb}
By looking at \eqref{eq:one} and \eqref{eq:two} in the proof of Theorem~\ref{thm:main1}, we find 
\[
\lambda_{\Lip}(Y)=\sup_{F\subset Y } \lambda_{\Lip}(F, Y).
\]
where the supremum is taken over all finite subsets \(F\subset Y\).
This identity has also been obtained by Brudnyi and Brudnyi in \cite[Corollary 2.3]{MR2288741} and is a crucial tool in their proof of \cite[Theorem 1]{MR2340707}.
\end{remark}
To conclude this section we give the proof of Theorem~\ref{thm:main2}.

\begin{proof}[Proof of Theorem~\ref{thm:main2}]
From Theorem~\ref{thm:main1} we get \[
\absolutE(X)\geq \sup_{Y\supset X} \lambda_{\Lip}(X,Y),
\]
and by the moreover part of Lemma~\ref{lem:reduction-to-finite-sets}, we obtain
\[
\absolutE(X)\leq \sup_{F\supset X} \nu(X,F),
\]
where the supremum is taken over all finite metric spaces \(F\) containing \(X\).
Hence, thanks to Lemma~\ref{lem:finishLine},
\[
\absolutE(X)=\sup_{F} \lambda_{\Lip}(X,F)=\sup_{F} \lambda(\mathcal{F}(X),\mathcal{F}(F))
\]
and thus \eqref{eq:main-theorem14} follows. To finish the proof, we must show \eqref{eq:simple-consequence}. Let \(\iota\colon X\to c_0\) be an isometric embedding and denote by \(E(X)\) the injective hull of \(X\). 
By virtue of Lemma~\ref{lem:co-is-finitely-injective}, there is a \(1\)-Lipschitz extension \(\bar{\iota}\colon E(X)\to c_0\) of \(\iota\). Notice that \(\nu(X, E(X))\leq \nu\bigl(\iota(X), \bar{\iota}\bigl(E(X)\bigr)\bigr)\), and so
\begin{equation}\label{eq:5}
\nu(X, E(X))\leq \nu(\iota(X), c_0).
\end{equation}
Lemma~\ref{lem:finishLine} tells us that \(\nu(\iota(X), c_0)=\lambda\bigl(\mathcal{F}(\iota(X)), \mathcal{F}(c_0)\bigr)\).
Hence, on account of \(\absolutE(X)=\nu(X, E(X))\), using \eqref{eq:5}, we find that \(\absolutE(X)=\lambda\bigl(\mathcal{F}(\iota(X)), \mathcal{F}(c_0)\bigr)\), concluding the proof.
\end{proof}

\section{Absolute Lipschitz extendability}\label{sec:four}

In this section, we prove the propositions appearing in Section~\ref{sec:applications}. The following lemma is the main tool in our proof of Proposition~\ref{prop:bounds-ae-sep}.

\begin{lemma}\label{lem:n-pods}
Let \(X\) be a non-empty metric space. Suppose that there is a constant \(C>0\) such that \(d(x,x')=2C\) for all \(x\), \(x'\in X\). Then 
\begin{equation}\label{eq:discrete-est}
\absolutE(X)\leq 2-2/\abs{X}
\end{equation}
with equality  whenever \(X\) is finite.
\end{lemma}

\begin{proof}
We may suppose that \(\abs{X}\geq 2\) and we fix an index set \(I\) such that \(X=\{x_i : i\in I\}\).
Further, without loss of generality we may assume that \(C=1\). Let \(W\) denote the metric space obtained by gluing the closed intervals \(I_i=[0,1]\subset \R\) along their origins. See Figure~\ref{fig:four-2} for an illustration when \(\abs{X}=5\).
\begin{figure}[h]
\centering
\includegraphics[scale=0.45]{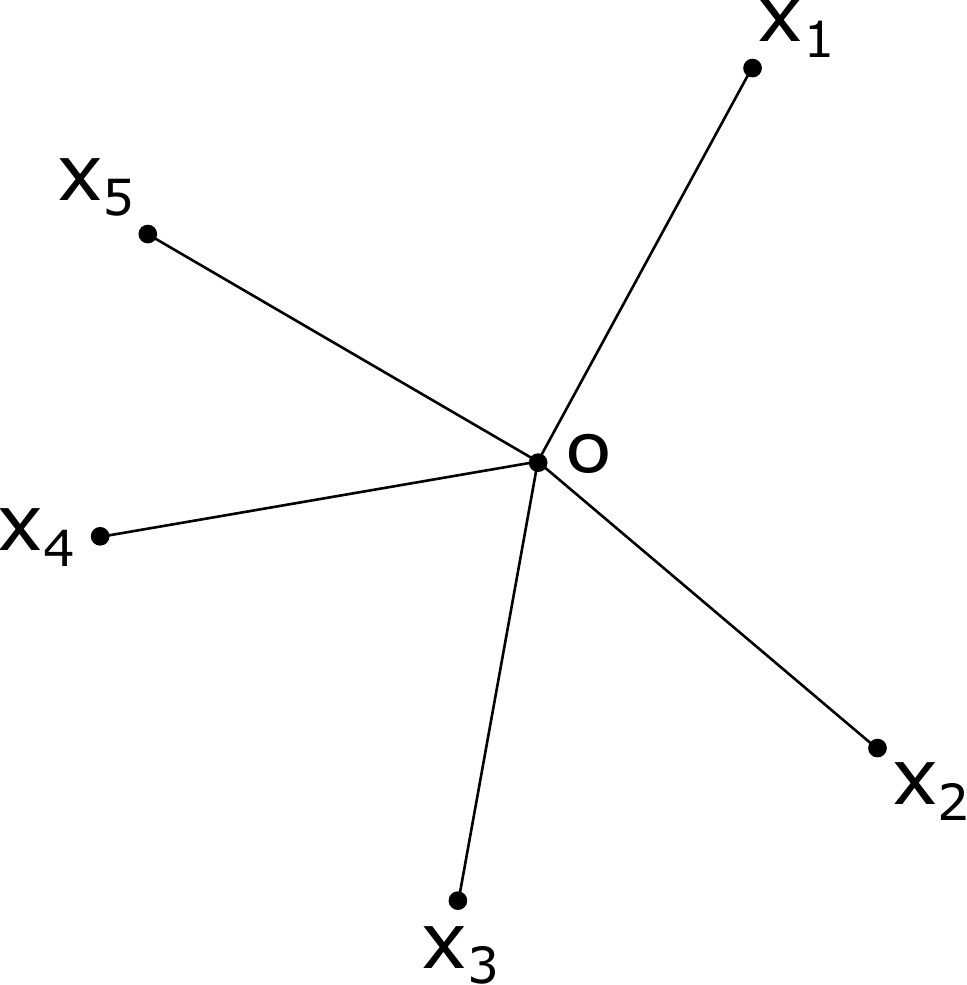}
\caption{Illustration of \(W\) when \(X=\{x_1, \dots, x_5\}\).}
\label{fig:four-2}
\end{figure}
Clearly, \(W\) is a complete \(\R\)-tree with internal vertex \(o\) and leaves which we identify with \(X\). Every complete \(\R\)-tree is injective and so, by Lemma~\ref{lem:reduction-to-inj-hull}, \(\absolutE(X)=\nu(X, W)\). 

We set \(Y\coloneqq X\cup \{o \}\subset W\).  It is easy to check that if \(f\colon Y\to E\) is a \(1\)-Lipschitz map, then the map \(\bar{f}\colon W\to E\)
which on every edge \([o,x]\subset W\) is given by \(w\mapsto (1-t) f(o)+t f(x)\) with \(t=d(o,w)/d(o,x)\) is a \(1\)-Lipschitz extension of \(f\). Therefore, \(\nu(Y, W)=1\) and we infer \(\absolutE(X)=\nu(X, Y)\).

By \cite[Theorem 1.1]{MR3898189}, any \(L\)-Lipschitz map between complete metric spaces can always be extended to one additional point such that the resulting map is \(2L\)-Lipschitz. As a result, \(\nu(X,Y)\leq 2\), which gives \eqref{eq:discrete-est} when \(X\) is infinite. 

Suppose now that \(X=\{x_1, \dots, x_n\}\) with \(n\geq 2\) and consider \(x_1\) as basepoint. 
Let \(f\colon Y\to \ell_1^n\) denote the map defined by \(f(x_1)=0\), \(f(x_i)=e_1-e_i\) for \(i=2, \dots, n\) and \(f(o)=e_1\). 
Due to Lemma~\ref{lem:explicit-isometry} the map \(\beta_f\colon \mathcal{F}(Y)\to \ell_1^n\) is a linear isometry. 
By construction, 
\(\beta_f(\mathcal{F}(X))=H\), where
\[
H\coloneqq\Big\{ v\in \ell_1^n : \sum_{i=1}^n v_i =0 \Big\},
\]
It is a classical fact due to Bohnenblust (see \cite[Section 5]{bohnenblust1938convex}) that  \(\lambda(H, \ell_1^n)=2-2/n\) for all \(n\geq 2\). As a result, we obtain \(\lambda(\mathcal{F}(X), \mathcal{F}(Y))=2-2/n\). By Lemma~\ref{lem:finishLine}, \(\nu(X, Y)=\lambda(\mathcal{F}(X), \mathcal{F}(Y))\), and so, by the above, 
\[
\absolutE(X)=\nu(X, Y)=\lambda(\mathcal{F}(X), \mathcal{F}(Y))=2-2/n,
\]
as desired.  
\end{proof}

Now we are in position to prove Proposition~\ref{prop:bounds-ae-sep}.

\begin{proof}[Proof of Proposition~\ref{prop:bounds-ae-sep}]
Suppose \(\rho(X)\neq \infty\), where 
\[
\rho(X)\coloneqq \frac{\diam X}{\sep X}.
\]
Let \(\delta\colon X\times X\to \R\) denote the metric given by \(\delta(x,x')\coloneqq \rho(X)\) for all distinct points \(x\), \(x'\in X\). We put \(X_\delta=(X, \delta)\). The identity map \(\id\colon (X,d)\to (X, \delta)\) is \(\rho(X)\)-bilipschitz. Consequently, using Lemma~\ref{lem:bilip-esti}, we find 
\begin{equation}\label{eq:distortion-1}
\absolutE(X)\leq \rho(X) \absolutE(X_\delta).
\end{equation}
As \(\delta\) is a discrete metric, Lemma~\ref{lem:n-pods} tells us that \(\absolutE(X_\delta)\leq 2-2/\abs{X}\). By combining this inequality with \eqref{eq:distortion-1}, we conclude \(\absolutE(X)\leq \rho(X)(2-2/\abs{X})\), as desired.
\end{proof}
Next, we proceed with the proof of Proposition~\ref{prop:ae3}.
\begin{proof}[Proof of Proposition \ref{prop:ae3}]
As pointed out in the introduction, \(\absolutE(n) \leq \lambda_{n-1}\). 
Thus, since \(\lambda_2=\frac{4}{3}\), as established by Chalmers and Lewicki in \cite{BruceL2010}, it follows that \(\absolutE(3)\leq \frac{4}{3}\). Now, Lemma~\ref{lem:n-pods} tells us that
for a three-point metric space \(X\) with \(d(x,x')=1\) for all distinct \(x\), \(x'\in X\), one has \(\absolutE(X)=2-\frac{2}{3}=\frac{4}{3}\). This completes the proof. 
\end{proof}

We are left to establish Proposition~\ref{prop:ae4}, whose proof is more involved than the previous ones. We follow the proof strategy outlined in Section~\ref{sec:applications}. To obtain the upper bound, we  need the following proposition:

\begin{proposition}\label{prop:injectiveHull}
For every four-point metric space \(X\), there exists a metric space \(Y\supset X\) with \(\abs{Y}\leq 8\) such that \(\absolutE(X)=\nu(X, Y)\).
\end{proposition}

\begin{proof}
Let \(X=\{x_1, \dots, x_4\}\) be a metric space consisting of four points. 
By relabeling the points if necessary, we can assume that
\begin{align*}
d(x_1, x_3)+d(x_2, x_4)&\geq d(x_1, x_2)+d(x_3, x_4),  \\
d(x_1, x_3)+d(x_2, x_4)&\geq d(x_1, x_4)+d(x_2, x_3).
\end{align*} 
By a result due to Dress (see \cite[Paragraph 1.16]{MR753872}), it follows that the distance matrix of \(X\) is of the following form:
\begin{equation}\label{eq:distance}
\begin{pmatrix}
0 & a+\ell+b & a+\ell+w+c & a+w+d \\
a+\ell+b & 0 & b+w+c & b+\ell+w+d  \\
a+\ell+w+c & b+w+c & 0 &c+\ell+d \\
a+w+d & b+\ell+w+d & c+\ell+d & 0
\end{pmatrix},
 \end{equation}
where \(a\), \(b\), \(c\), \(d\geq 0\) and \(\ell \geq w \geq 0\). Here, the \((i,j)\)-th entry of the above matrix is equal to \(d(x_i, x_j)\). Dress (see \cite[Paragraph 1.16]{MR753872}) showed that the injective hull \(E(X)\) of \(X\) is isometric to the metric space depicted in Figure~\ref{fig:four} considered as a subset of \(\ell_1^2=(\R^2, \norm{\cdot}_1)\). 
\begin{figure}[t]
\centering
\includegraphics[scale=0.45]{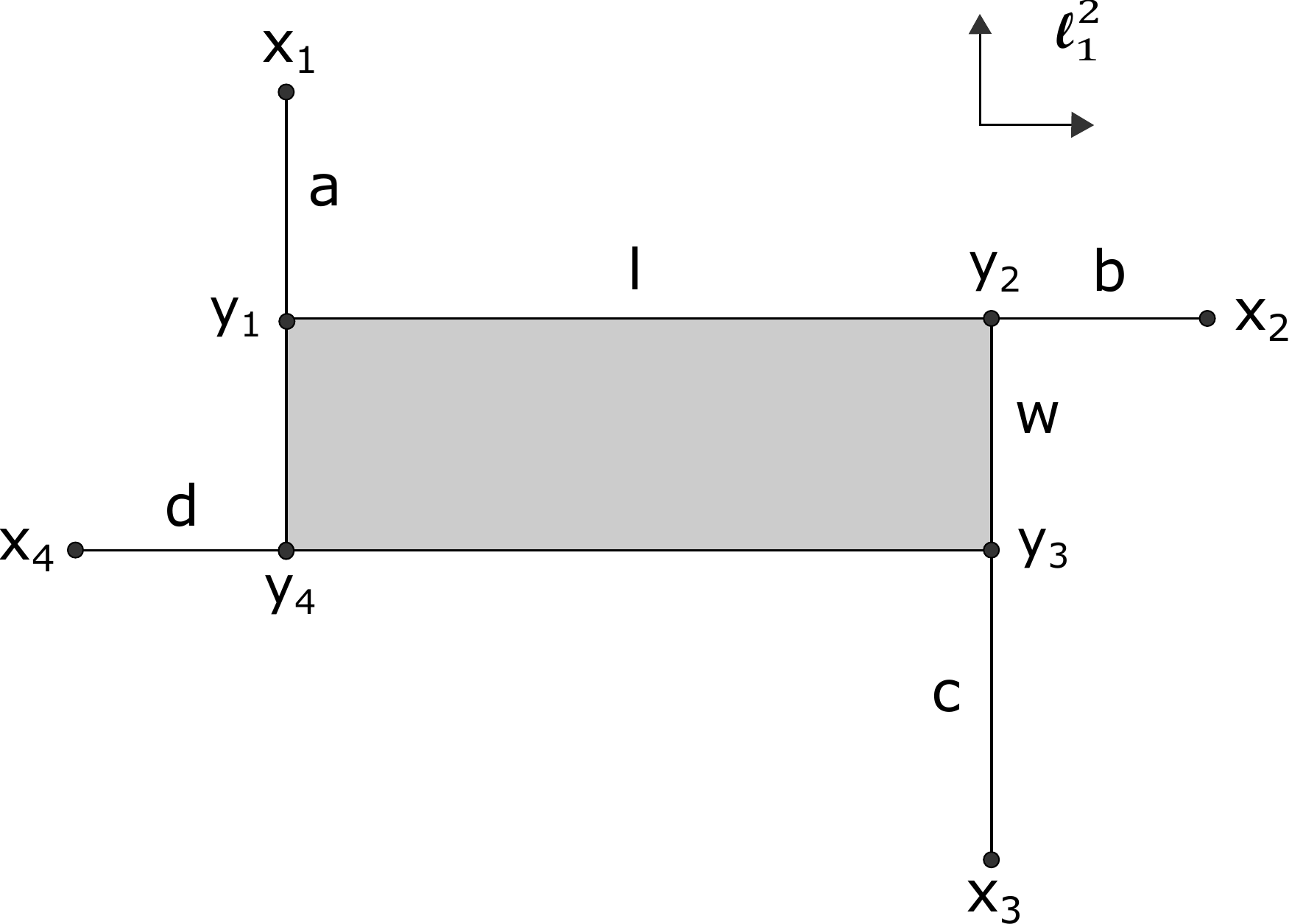}
\caption{The injective hull \(E(X)\subset \ell_1^2\) of a four-point metric space \(X=\{x_1, \dots, x_4\}\) with distance matrix \eqref{eq:distance}.}
\label{fig:four}
\end{figure}
Let \(y_1\in E(S)\subset \ell_1^2\) denote the unique point such that \(d(x_1,y_1)=a\), and define the points \(y_2\), \(y_3\), \(y_4\in E(X)\) analogously. We set \(Y\coloneqq X\cup\{y_1, \dots, y_4\}\subset E(X)\). 

We claim that \(\nu(X, Y)=\absolutE(X)\). As \(\absolutE(X)=\nu(X, E(X))\), to establish the claim it suffices to show that \(\nu(X, E(X))\leq \nu(X, Y)\). Let \(f\colon X\to E\) be a \(1\)-Lipschitz function to a Banach space \(E\) and denote by \(\bar{f}\colon Y\to E\) a \(K\)-Lipschitz extension of \(f\) to \(Y\), where \(K\coloneqq \nu(X,Y)\). We consider \(E(X)\) as a subset of \(\ell_1^2\) such that the sides of the rectangle with vertices \(y_i\) are parallel to the coordinate axes. We define the map \(\bar{f}\colon E(X)\to E\) as follows:
\begin{itemize}
    \item If \(y=(1-t)x_i+t y_i\) for some \(t\in [0,1]\), then we set \[
    \bar{f}(y)\coloneqq (1-t)\bar{f}(x_i)+t\bar{f}(y_i).\] 
    \item If \(y=\alpha_1 y_1+\alpha_2 y_2+\alpha_3 y_3\) with \(\alpha_i\geq 0\) and \(\alpha_1+\alpha_2+\alpha_3=1\), then we set
    \[
    \bar{f}(y)\coloneqq \alpha_1 \bar{f}(y_1)+\alpha_2 \bar{f}(y_2)+\alpha_3 \bar{f}(y_3).
    \]
    \item If \(y=\beta_1 y_1+\beta_3 y_3+\beta_4 y_4\) with \(\beta_i\geq 0\) and \(\beta_1+\beta_3+\beta_4=1\), then we set
    \[
    \bar{f}(y)\coloneqq \beta_1 \bar{f}(y_1)+\beta_3 \bar{f}(y_3)+\beta_4 \bar{f}(y_4). 
    \]
\end{itemize}
The map \(\bar{f}\) is well-defined and an extension of \(f\). In the following, we show that \(f\) is \(K\)-Lipschitz. Notice that if \(y=\beta_1 y_1+\beta_3 y_3+\beta_4 y_4\) and \(y'=\beta_1' y_1+\beta_3' y_3+\beta_4' y_4\), then
\[
y-y'=(\beta_1-\beta_1')(y_1-y_4)+(\beta_3-\beta_3')(y_3-y_4)
\]
and therefore 
\begin{equation}\label{eq:distance-2}
d(y,y')=\norm{y-y'}_1=\abs{\beta_1-\beta_1'}w+\abs{\beta_3-\beta_3'}\ell.
\end{equation}
Clearly,
\[
\norm{\bar{f}(y)-\bar{f}(y')}\leq \abs{\beta_1-\beta_1'}\,\, \norm{\bar{f}(y_1)-\bar{f}(y_4)}+\abs{\beta_3-\beta_3'}\,\,\norm{\bar{f}(y_3)-\bar{f}(y_4)};
\]
hence, by virtue of \eqref{eq:distance-2}, we have
\(\norm{\bar{f}(y)-\bar{f}(y')}\leq K d(y,y')\) for all points \(y\), \(y'\) contained in the convex hull of \(y_1, y_3, y_4\). 

Analogously, we find that \(\bar{f}\) is \(K\)-Lipschitz on the convex hull of the points \(y_1, y_2, y_3\). A straightforward computation shows that \(\bar{f}\) is \(K\)-Lipschitz on \(E(X)\), as desired. Since \(f\) and \(E\) were arbitrary, it follows that \(\nu(X, E(X))\leq \nu(X, Y)\), as was to be shown. 
\end{proof}

Finally, we give the proof of Proposition~\ref{prop:ae4}.

\begin{proof}[Proof of Proposition~\ref{prop:ae4}]
Let \(X\) be any four-point metric space. By Proposition~\ref{prop:injectiveHull}, there exists a metric space \(Y\supset X\) such that \(\abs{Y}\leq 8\) and \(\absolutE(X)=\nu(X, Y)\). Lemma~\ref{lem:finishLine} tells us that \(\nu(X, Y)=\lambda(\mathcal{F}(X), \mathcal{F}(Y))\). Consequently, as \(\dim \mathcal{F}(X)=3\) and \(\dim \mathcal{F}(Y)\leq 7\), we have \(\absolutE(X) \leq \bar{\lambda}(3,7)\),
where \(\bar{\lambda}(n,d)\) is defined as in \eqref{eq:def-of-bar-lambda}.
Using the upper bound \eqref{eq:koenig}, we arrive at
\[
\absolutE(4)\leq \frac{3+6 \sqrt{2}}{7},
\]
as desired.
Next, we deal with the lower bound. We abbreviate 
\[
\alpha\coloneqq\frac{1}{\sqrt{2}}+\frac{1}{2} \quad\text{ and }\quad\beta\coloneqq \frac{1}{\sqrt{2}}-\frac{1}{2}.
\]
We define the points \(x_1, x_2, x_3, x_4, y_1, y_2\in  \ell^2_\infty\) as follows:
\begin{align*}
x_1&\coloneqq(-\alpha, -1),& x_2&\coloneqq(-\alpha, 1),& x_3&\coloneqq(\alpha,1),& x_4&\coloneqq(\alpha, -1),\\
y_1&\coloneqq(-\beta, 0),& y_2&\coloneqq(\beta,0).
\end{align*}
The points \(x_1, \dots, x_4\) are the vertices of a rectangle with side lengths \(2\) and \(\sqrt{2}+1\). We set \(X\coloneqq\{x_1, \dots, x_4\}\subset \ell_\infty^2\) and \(Y\coloneqq X\cup\{ y_1, y_2\}\subset \ell_\infty^2\). By Theorem~\ref{thm:main2}, 
\[
\absolutE(X)\geq \lambda(\mathcal{F}(X), \mathcal{F}(Y)).
\]  
In the following, we show that \(\lambda(\mathcal{F}(X), \mathcal{F}(Y))\geq (5+4\sqrt{2})/7\). Let \(T=(Y, E)\) denote the graph with vertex set \(Y\) and edge set \(E=\{f_1, \dots, f_5\}\), where 
\begin{align*}
f_1&\coloneqq \{x_1, y_1\}, & f_2&\coloneqq \{x_2, y_1\}, & f_3&=\{y_1, y_2\}, \\
f_4&\coloneqq\{y_2, x_3\}, & f_5&\coloneqq \{y_2, x_4\}. 
\end{align*}
By construction, \(T\) is a tree, and every  \(x_i\) is a leaf of \(T\) and \(y_1\) and \(y_2\) are internal vertices of degree 3. Let \(d_T\) denote shortest-path metric on \(T\) induced by the weight function \(\omega\colon E\to (0, \infty)\) defined by \(\omega(\{x,y\})\coloneqq d(x,y)\). The identity map \(\id\colon (Y, d)\to (T, d_T)\) is an isometry. Fix \(x_1\) as a basepoint. Clearly, \(B=\{b_1, \dots, b_5\}\) with \(b_i=\delta(x_{i+1})\), for \(i=1,2,3\), and \(b_i=\delta(y_{i-3})\), for \(i=4,5\), is a basis of \(\mathcal{F}(T)\). Using this basis, let \(L\colon \mathcal{F}(T)\to \ell_1^5\) denote the linear map induced by the matrix 
\begin{equation*}
M\coloneqq\left[
\begin{array}{ccccc}
 1 & 1 & 1 & 1 & 1 \\
 1 & 0 & 0 & 0 & 0 \\
 0 & \gamma & \gamma & 0 & \gamma \\
 0 & 1 & 0 & 0 & 0 \\
 0 & 0 & 1 & 0 & 0 \\
\end{array}
\right],
\end{equation*}
where \(\gamma\coloneqq \sqrt{2}-1\). Notice that the map \(f\colon T\to \ell_1^5\) defined by \(f\coloneqq L\circ\delta\) satisfies the assumptions of Lemma~\ref{lem:explicit-isometry}.
Consequently, as \(L=\beta_f\), which is due to the uniqueness of \(\beta_f\), it follows from Lemma~\ref{lem:explicit-isometry} that \(L\) is a linear isometry. 

Let \(v_i\in \ell_1^5\) denote the \(i\)th column of the matrix \(M\). By definition of \(L\), one has \(L(\mathcal{F}(X))=E\), where \(E\subset \ell_1^5\) is the linear span of the vectors \(v_1\), \(v_2\) and \(v_3\). Hence, as \(L\) is a linear isometry, \(\lambda(\mathcal{F}(X), \mathcal{F}(Y))=\lambda(E, \ell_1^5)\). We set
\begin{equation}
    S\coloneqq\left[
\begin{array}{ccccc}
 1 & 1 & 1 & 1 & 1 \\
 1 & 1 & -1 & -1 & -1 \\
 1 & -1 & 1 & 1 & 1 \\
 1 & -1 & 1 & 1 & -1 \\
 1 & -1 & 1 & -1 & 1 \\
\end{array}
\right]
\end{equation}
and denote by \(D\in \mathcal{M}_5(\R)\) the diagonal matrix with diagonal \((\nu, \nu, \mu, \nu, \nu)\), where \(\mu\coloneqq (5-3\sqrt{2})/7\) and \(\nu\coloneqq (1-\mu)/4\). Via a direct computation, one can show that
\(u_1\coloneqq v_2+v_3-v_1\) is an eigenvector of \(DS\) with eigenvalue \(\lambda_1\coloneqq(\sqrt{2}+3)/7\), and the independent vectors \(u_2\coloneqq v_3-v_2\) and \(u_3\coloneqq v_1\),  are  eigenvectors of \(DS\) with eigenvalues \(\lambda_2=\lambda_3\coloneqq (3 \sqrt{2}+2)/14\). 

Since \(u_1, u_2, u_3\) is a basis of \(E\), we find that \(AP=PAP\), where \(A\coloneqq DS\) and \(P\colon \ell_1^5\to E\) denotes the transformation matrix of the orthogonal projection from \(\R^5\) onto \(E\subset \R^5\).
Notice that \(\Tr(D)=1\) and thus, by \eqref{eq:1-nuclear-norm}, it follows that \(\nu_1(A)=1\). By Proposition~\ref{prop:rel-pro-via-trace-duality},
\[
\lambda(E, \ell_1^5)\geq \Tr(AP)=\sum_{i=1}^3 \frac{1}{\norm{u_i}^2_2} \Tr( A u_i u_i^T)=\sum_{i=1}^3 \lambda_i=\frac{5+4\sqrt{2}}{7}, 
\]
where we have used that the vectors \(u_i\) are orthogonal to each other. 
Using the identity \(\lambda(\mathcal{F}(X), \mathcal{F}(Y))=\lambda(E, \ell_1^5)\), we conclude that
\[
\absolutE(4)\geq \absolutE(X)\geq \lambda(\mathcal{F}(X), \mathcal{F}(Y))\geq\frac{5+4\sqrt{2}}{7}.
\]
This completes the proof. 
\end{proof}

\begin{remark}
Given a matrix \(A\), the matrix \(\textrm{Sgn}(A)=(b_{ij})\) defined by \(b_{ij}\coloneqq \sgn(a_{ij})\) is called \textit{sign pattern matrix} of \(A\). The matrix \(S\) appearing in the proof of Proposition~\ref{prop:ae4} is the sign pattern matrix of the matrix \(P\). In general, given a linear projection \(P\colon \ell_1^d\to E\), good candidates for a matrix \(A\in \mathcal{M}_d(\R)\) with  \(\Tr(AP)=\lambda(E, \ell_1^d)\) are matrices of the form \(A=D \textrm{Sgn}(P)\), where \(D\) is a diagonal matrix with positive entries. 
\end{remark}

\section{Projection constants and linear programming}

\subsection{Banach space theory}

In the following, we recall standard concepts from Banach space theory. For each integer \(d \geq 1\) let \(\mathcal{M}_{d}(\mathbb{R})\) denote the real vector space of all \(d\times d\) matrices with real entries. For every linear map \(f\colon \mathcal{M}_{d}(\mathbb{R})\to \R\) there exists a unique matrix \(A\in \mathcal{M}_{d}(\mathbb{R})\) such that
\begin{equation*}
f(X)=\Tr(AX) \quad \text{ for all } X\in \mathcal{M}_{d}(\mathbb{R}).
\end{equation*}
For every Banach space \(E=(\R^d, \norm{\cdot}_{_E})\) we let \(\mathcal{L}(E)=(\mathcal{M}_d(\R), \norm{\cdot}_{\mathcal{L}(E)})\) denote Banach space of all linear operators from \(E\) to \(E\) equipped with the operator norm \(\norm{\cdot}_{\mathcal{L}(E)}\). The \textit{\(1\)-nuclear norm} \(\nu_1(A)\) of \(A\in \mathcal{L}(E)\) is defined as follows
\[
\nu_1(A)\coloneqq\inf\Big\{ \sum_{i=1}^m \norm{f_i}_{_{E^\ast}} \norm{x_i}_{_E} : m\geq 1, A=\sum_{i=1}^m f_i \otimes x_i\Big\}. 
\]
 It is well-known that the dual space of \(\mathcal{L}(E)\) can be identified with \((\mathcal{M}_d(\R), \nu_1(\cdot))\), and so
 \begin{equation*}
\nu_1(A)=\max\big\{ \Tr(AX) \,:\,  X\in \mathcal{M}_{d}(\mathbb{R}) \textrm{ with }\norm{X}_{\mathcal{L}(E)}=1\big\} 
\end{equation*}
for all \(A\in \mathcal{M}_d(\R)\). If \(E=\ell_1^d\), then 
\begin{equation}\label{eq:1-nuclear-norm}
\nu_1(A)=\sum_{i=1}^d \max_{j} \, \abs{a_{ij}}
\end{equation}
for all \(A\in \mathcal{M}_d(\R)\). 

Given a Banach space \(E\), we denote by \(\ext B_E\) the set of extreme points of the closed unit ball \(B_E\) of \(E\). A finite-dimensional Banach space \(E\) is said to be \textit{polyhedral}, 
if \(\ext B_{E^\ast}\) is finite.  
If \(E\) is polyhedral and \(S\coloneqq\{f_1, \dots, f_d\}\subset \ext B_{E^\ast}\) is a subset such \(S\cup -S=\ext B_{E^\ast}\) and \(\abs{\ext B_{E^\ast}}=2d\), then \(\alpha\colon E\hookrightarrow \ell_\infty^d\) defined by \(x\mapsto (f_i(x))\) is a linear isometric embedding. Conversely, every finite-dimensional linear subspace of \(\ell_\infty^d\) is polyhedral.
If \(F=(\R^d, \norm{\cdot}_{_F})\) is a polyhedral Banach space, then \(\mathcal{L}(F)\) is polyhedral as well. In fact, \(\ext B_{\mathcal{L}(F)^\ast}\) is equal to \(\ext B_{F^\ast} \otimes \ext B_F\) (see \cite{MR682665}). 

\subsection{Computation of projection constants via linear programming}

The aim of this section is to establish a characterization of \(\lambda(E,F)\) as the optimal value of a certain linear programming problem (see Lemma~\ref{lem:lin-prog}). 

Let \(F=(\R^d,\norm{\cdot}_F)\) be a polyhedral Banach space and \(E\subset F\) a nontrivial linear subspace that is not equal to \(F\). Suppose that \(u_r\in \R^d\), \(r=1, \dots, n\), is an orthonormal basis of \(E\subset \R^d\) and \(v_s\in\R^d\), \(s=1, \dots, d-n\), an orthonormal basis of its orthogonal complement. 
Fix an enumeration \(\{B_1, \dots, B_p\}\) of 
\[
\big\{u_{r^{\phantom{1}}}\hspace{-0.25em}u_{r'}^T\big\}\cup\big\{v_{\hspace{-0.1em}{s^{\phantom{1}}}}\hspace{-0.25em}v_{\hspace{-0.1em}s'}^T\big\}\cup\big\{u_{r^{\phantom{1}}}\hspace{-0.35em}v_{\hspace{-0.1em}s}^T\big\}
\]
such that \(B_i=u_i u_i^T\) for all \(i=1, \dots, n\) and fix an enumeration \(\{E_1, \dots, E_{q}\}\) of the extreme points of the closed unit ball of \(\mathcal{L}(F)\). 

\begin{lemma}\label{lem:lin-prog}
We define the matrix \(A\in \mathcal{M}_{p\times q}(\R)\)  via \(a_{ij}=\Tr(B_iE_j)\) and the vector \(b\in \R^p\) by setting \(b_i=1\) for all \(i=1, \dots, n\) and \(b_i=0\) otherwise. Let \(j\in \R^q\) denote the all-ones vector. Then the linear programming problem
\begin{align*}\label{eq:programming}
\text{minimize} \quad
j^Tx& \\
\text{subject to} \quad
Ax &= b \nonumber \\
x &\geq 0 \nonumber
\end{align*}
is solvable and its optimal value is equal to \(\lambda(E,F)\). 
\end{lemma}

We suggest the book \cite{matousek2007understanding} for an introduction to linear programming. A few remarks are in order. 
\begin{itemize}
\item Lemma~\ref{lem:lin-prog} is particularly useful when \(F=\ell_\infty^d\). In this case \(\lambda(E, F)=\lambda(E)\) for every subspace \(E\subset F\). This is a direct consequence of the fact that \(\ell_\infty^d\) is an injective Banach space for every \(d\geq 1\). 
\item In the special case when \(F=\ell_\infty^d\), the set of extreme points \(\ext B_{\mathcal{L}(F)}\) can be characterized as follows: \(M \in \mathcal{M}_d(\R)\) is contained in \(\ext B_{\mathcal{L}(F)}\) if and only if for each \(i\in \{1, \dots, d\}\) there exists a unique \(j\in\{1, \dots, d\}\) such that \(\abs{m_{ij}}=1\) and \(m_{ik}=0\) for all \(k\neq j\). Hence, the set \(\ext B_{\mathcal{L}(\ell_\infty^d)}\) contains exactly \(2^{d+1} d^d\) matrices. 

\item For an arbitrary polyhedral Banach space \(F\), computing the extreme points of the closed unit ball of \(\mathcal{L}(F)\) is not an easy task. 
Since \(\ext B_{\mathcal{L}(F)^\ast}=\ext B_{F^\ast} \otimes \,\ext B_F\), the \(V\)-representation of the convex polytope
\[
\mathcal{P}\coloneqq B_{\mathcal{L}(F)^\ast}\subset \mathcal{M}_d(\R)
\]
is known, and a general strategy for determining \(\ext B_{\mathcal{L}(F)}\) is to compute the \(H\)-representation of \(\mathcal{P}\). 
\end{itemize}

The following formula for \(\lambda(E,F)\) is essentially known (see \cite[Lemma 1]{konig1983finite}). 
For the convenience of the reader, we give an elementary and detailed proof of it at the end of the subsection. 

\begin{proposition}\label{prop:rel-pro-via-trace-duality}
Let \(F=(\R^d, \norm{\cdot})\) be a polyhedral Banach space and \(E\subset F\) a linear subspace. Then
\begin{equation}\label{eq:toShoww}
\lambda(E, F) = \max\big\{ \Tr(AP): A\in\mathcal{M}_{d}(\mathbb{R}),\, \nu_1(A)=1 \textrm{ and } AP=PAP \,\big\},
\end{equation}
where \(P\in \mathcal{M}_{d}(\R) \) is the transformation matrix of the orthogonal projection from \(\R^d\) onto \(E\subset \R^d\). 
\end{proposition}

Now, Lemma~\ref{lem:lin-prog} follows from the proposition above and the strong duality theorem of linear programming.

\begin{proof}[Proof of Lemma~\ref{lem:lin-prog}]
Notice that \(B\in \mathcal{M}_d(\R)\) satisfies \(BP=PBP\) if and only if 
\[
B=\sum_{i=1}^p x_i B_i
\]
for some real numbers \(x_1, \dots, x_p\in \R\). For any such matrix \(B\) we have
\[
\Tr(BP)=\langle b, x\rangle_{\R^p}=b^Tx,
\]
where \(x\coloneqq (x_1, \dots, x_p)\). Further, notice that \(\nu_1(B)\leq 1\) if and only if \(A^T x\leq j\). Hence, using Proposition~\ref{prop:rel-pro-via-trace-duality}, we find that the optimal value of the linear programming problem
\begin{align*}
\text{maximize} \quad
b^T x& \\
\text{subject to} \quad
A^Tx &\leq j \nonumber
\end{align*}
is equal to \(\lambda(E, F)\). Now,  the conclusion of the lemma follows directly from the strong duality theorem of linear programming (see, for example, \cite[p. 85]{matousek2007understanding}). 
\end{proof}

\begin{proof}[Proof of Proposition~\ref{prop:rel-pro-via-trace-duality}]
Let \(A\in\mathcal{M}_{d}(\mathbb{R})\) be such that \(AP=PAP\) and let \(Q\in  \mathcal{M}_{d}(\mathbb{R})\) be a projection matrix with range \(E\).
We compute
\begin{equation*}
\begin{split}
\Tr(AP)&=\Tr(AQP)=\Tr(PAQ)\\ 
&=\Tr(PAPQ)=\Tr(APQ)=\Tr(AQ)\leq \nu_1(A)\,\norm{Q}_{\mathcal{L}(F)}.
\end{split}
\end{equation*} 
Hence, the left hand side of \eqref{eq:toShoww} is greater than or equal to the right hand side. In the following, we establish the other inequality. 

Clearly, \eqref{eq:toShoww} is true if \(E=\{0\}\) or \(E=F\), therefore we may suppose that \(\dim(E)\neq 0, d\). Let \(u_r\in\R^d\), \(r=1, \dots, n\), be an orthonormal basis of \(E\subset \R^d\) and \(v_s\in\R^d\), \(s=1, \dots, d-n\), an orthonormal basis of its orthogonal complement. 
Every matrix contained in the set \(P+H\), where \(H\) is the linear span of \(\big\{u_{r^{\phantom{1}}}\hspace{-0.35em}v_{\hspace{-0.1em}s}^T\big\}\), is a projection matrix with range \(E\). Let \(\textrm{Proj}\subset \mathcal{M}_{d}(\mathbb{R}) \) denote the linear span of \(P+H\) and fix a decomposition \(\mathcal{M}_{d}(\mathbb{R})=\textrm{Proj} \oplus \textrm{Aux}\), where \(\textrm{Aux}\subset \mathcal{M}_{d}(\mathbb{R}) \) is a linear subspace.
We set 
\begin{equation*}
\textrm{Proj}^\ast\coloneqq \big\{ A\in \mathcal{M}_{d}(\mathbb{R}) : \Tr(AX)=0 \textrm{ for all } X\in \textrm{Aux} \big\}.
\end{equation*}
We claim that there exists \(A\in\textrm{Proj}^\ast\), \(A \neq 0\), such that \(AP=PAP\). 

The linear map  \(L\colon \mathcal{M}_{d}(\mathbb{R}) \to \mathcal{M}_{d}(\mathbb{R})\) defined by \(X\mapsto XP-PXP\) satisfies
\(L|_{W_0}=0\) and \(L|_{W_1}=\textrm{id}|_{W_1}\), 
where \( W_0\subset \mathcal{M}_{d}(\mathbb{R})\) is the linear span of the matrices 
\(
\big\{u_{r^{\phantom{1}}}\hspace{-0.25em}u_{r'}^T\big\}, \big\{v_{\hspace{-0.1em}{s^{\phantom{1}}}}\hspace{-0.25em}v_{\hspace{-0.1em}s'}^T\big\}  \textrm{ and } \big\{u_{r^{\phantom{1}}}\hspace{-0.35em}v_{\hspace{-0.1em}s}^T\big\}
\)
and \( W_1\subset \mathcal{M}_{d}(\mathbb{R})\) the linear span of the matrices \(\{v_{\hspace{-0.1em}s^{\phantom{1}}}\hspace{-0.25em}u_{r}^T\}\). Hence, the kernel of \(L\) has dimension \(d^2-n(d-n)\), and since the dimension of \(\textrm{Proj}^\ast \subset \mathcal{M}_{d}(\mathbb{R})\) is equal to \(n(d-n)+1\), the matrix \(A\) with the desired properties exists. 

Now, fix \(A\in\textrm{Proj}^\ast\cap W_0\), \(A\neq 0\),  and define the linear map \(f_A\colon \textrm{Proj}\to \R\) via \(X\mapsto \Tr(AX)\). As \(A\in\textrm{Proj}^\ast\) and \(A\neq 0\), it follows that \(f_A\neq 0\).
In the following, we consider \(\textrm{Proj}\) as a linear subspace of \(\mathcal{L}(F)\) equipped with the induced norm. Clearly, there exists a matrix \(B\in\textrm{Proj}\) such that \(\norm{B}=1\) and \(\norm{f_A}=f_A(B)\). In particular, \(f_A(B)\neq 0\).

As \(AP=PAP\), it follows that \(Ax\in E\) for all \(x\in E\), which is in fact equivalent to \(AP=PAP\). Consequently,
\begin{equation}\label{eq:A-vanishes-on-H}
\Tr(A \hspace{0.1em}u_{r^{\phantom{1}}}\hspace{-0.35em}v_{\hspace{-0.1em}s}^T)=\langle A u_r, v_{\hspace{-0.1em}s}\rangle_{\R^d}=0
\end{equation}
for all \(r=1, \dots, n\) and all \(s=1, \dots, d-n\). This implies that \(\Tr(AX)=0\) for all \(X\in H\).
Hence, as \(f_A(B)\neq 0\), there is a projection matrix \(Q\in P+H\) and a scalar \(\alpha\neq 0\) such that \(B=\alpha Q\). Notice that \(\norm{f_A}=\alpha \Tr(AQ)\) and \(1=\norm{B}=\abs{\alpha}\cdot \norm{Q}\). Therefore,
\[
\norm{Q}=\Tr\Bigl( \tfrac{\sgn(\alpha) }{\norm{f_A}}\, A Q\Bigr),
\]
where \(\sgn(\alpha)\) denotes the sign of \(\alpha\). 

By the Hahn-Banach theorem there is a linear functional \(\bar{f}_A\colon \mathcal{L}(F)\to \R\) extending \(f_A\) such that \(\norm{\bar{f}_A}=\norm{f_A}\). Hence, by trace duality there exists \(\bar{A}\in \mathcal{M}_{d}(\mathbb{R})\) such that \(\bar{f}_A(X)=\Tr(\bar{A}X)\) for all \(X\in \mathcal{M}_{d}(\mathbb{R})\). Clearly, \(\nu_1(\bar{A})=\norm{\bar{f}_A}\) and \(\Tr(\bar{A}X)=\Tr(AX)\) for all \(X\in \textrm{Proj}\).
Using \eqref{eq:A-vanishes-on-H}, we find that \(\langle \bar{A} u_r, v_s\rangle_{\R^d}=0\); this implies that \(\bar{A}x\in E\) for all \(x\in E\) and thus \(\bar{A}P=P\bar{A} P\).

Letting \(A_0\coloneqq \bigl(\sgn(\alpha) /\nu_1(\bar{A})\bigr)\, \bar{A}\), it follows by the above that
\(\norm{Q}=\Tr(A_0 Q)\) and \(A_0P=PA_0P\). Hence, using that \(\Tr(A_0 P)=\Tr(A_0 Q)\) and \(\norm{Q}\geq \lambda(E, F)\), we conclude \(\Tr(A_0 P)\geq \lambda(E, F)\), as desired. This completes the proof.
\end{proof}

\printbibliography

\end{document}